\documentclass[11pt]{amsart}

\usepackage{latexsym}

\usepackage{amsmath,amssymb,amsthm,amsfonts,latexsym}
\usepackage{pstricks, pst-node, pst-text, pst-3d}
\usepackage[bookmarks]{hyperref}
\hypersetup{backref, colorlinks=true}

\let\=\partial

\newtheorem {pro}{Proposition}[section]
\newtheorem {thm}[pro]{Theorem}
\newtheorem {cor}[pro]{Corollary}
\newtheorem{lem}[pro]{Lemma}

\theoremstyle{definition}
 \newtheorem {rem}[pro]{Remark}
\newtheorem {dfn}[pro]{Definition}

\newtheorem {obs}[pro]{Observation}

\newcommand{\pb}{\overline{p}}

\newcommand{\qb}{\overline{q}}
\newcommand{\tb}{\overline{t}}

\newcommand{\R}{\mathbb{R}}

\newcommand{\N}{\mathbb{N}}
\newcommand{\C}{\mathbb{C}}
\newcommand{\K}{\mathbb{K}}
\newcommand{\kk}{\mathbb{K}}
\newcommand{\Pp}{\mathbb{P}}

\title {Geometry of polynomial mappings at infinity via intersection homology}



\author[A. and G. Valette]{Anna Valette and Guillaume Valette}

\address[A. Valette]{Instytut Matematyki Uniwersytetu
Jagiello\'nskiego, ul. S \L ojasiewicza, Krak\'ow, Poland}
\email{anna.valette@im.uj.edu.pl}
\address[G. Valette]
{Instytut Matematyczny PAN, ul. \'Sw. Tomasza 30, 31-027 Krak\'ow,
Poland} \email{gvalette@impan.pl}


\thanks{This research was done durring the authors stay in the Fields Institute in 2009}


\begin{document}
\maketitle

\begin{abstract}
For given polynomial map $F:\C^2\to\C^2$ with nonvanishing jacobian we asociate a variety whose homology or intersection homology describes
the geometry of singularities at infinity of this map.
\end{abstract}

\section{Introduction}
In the study of geometrical or topological properties of polynomial mappings, the set of points at which those maps fail to be proper plays an important role.
We call the smallest set $A$ such that the map $F:X\setminus F^{-1}(A)\to Y\setminus A $ is proper, the asymptotic set.  For instance, recall the famous Jacobian conjecture, in 1939 O.H. Keller \cite{kel}
asked whether any polynomial map $F:\C^n\to\C^n$ with nowhere vanishing Jacobian is a polynomial automorphism.  Since then many mathematicians tried to answer this question. In the topological approach, the main difficulty is to verify if such a map is proper. It is then natural to study the topology of  the asymptotic set.
In the 90's of the previous century Z. Jelonek studied properties of this set and obtained very important results. We recall briefly some of them in the next section.
The starting point of the study of  the asymptotic set is the simple observation that the asymptotic set of a given map $F:\kk^n\to \kk^m$ is the image under canonical projection of the $\overline{\hbox{graph}(F)}\setminus \hbox{graph}(F)$, where the closure is taken in the $\Pp^n\times\kk^m$, here and in the sequel $\kk$ stands for the field of complex or real numbers. This gives many important consequences. But if one is interested to keep the assumption of non-vanishing Jacobian,   compactifications have to be avoided. For instance a compactification drastically affect the volume of the  gvien Euclidian space.  If the jacobian is constant then  locally the mapping  $F$ preserves the volume. Thus, most of the information is lost in the compactification.

Our aim is to give a new approach for the Jacobian conjecture. We will construct a pseudomanifold $N_F$ associated to the given polynomial map $F:\C^n\to\C^n$. In the case $n=2$ we will then prove that the map $F$ with non-vanishing jacobian is not proper iff  the homology  or the intersection homology  of $N_F$ is nontrivial.

This provides a new and original approach of the Jacobian conjecture. In general, singularities of mappings are much more difficult to handle than singularities of sets. We reduce the study of an application $F:\C^n \to \C^n$ to the study of a singular semialgebraic set $N_F$.

Our approach is of metric nature. The significant advantege  is to be able to investigate the singularities, without looking at singularities themselves, but just at the regular points nearby.  For instance, it was proved in \cite{v} that the $L^\infty$ cohomology  of the regular locus of a given subanalytic pseudomanifold (which may be proved to be a subanalytic  bi-Lipschitz invariant of the regular locus) is a topological invariant of the pseudomanifold. Being able to work beside the singularities could prove to be useful to investigate singularities at infinity (where we hardly can go). In this paper we shall relate $L^\infty$ cohomology of the regular locus of the constructed pseudomanifold $N_F$ to the behavior of $F$ at infinity (Corollary \ref{cor homology propre}).  This result relies on the de Rham theorem proved in \cite{v}, which yields and isomorphism between $L^\infty$ cohomology and intersection cohmology in the maximal perversity.

The idea is to transfer the problem to differential geometry so as to reformulate the problem later in terms of differential equations. Most of the usual geometric approaches only emphasize that the Jacobian is nonzero rather than the fact that it is constant. 

It seems indeed more useful to find intersection homology than homology. The reason is that we are also able to prove that the set $N_F$ is a stratified pseudomanifold with only even codimensional strata. The main feature of intersection homology is that M. Goresky and R. MacPhserson were able to show in their fundamental paper that their theory satisfies Poincar\'e duality for stratified pseudomanifold and the case of pseudomanifolds stratified by even codimensional strata is the most important one if one is interested in the middle perversity. The case of the maximal perversity is also of importance since, as we said above,  it makes it possible to relate the geometry of $F$ at infinity to the $L^\infty$ cohomology of $N_F$.



\section{Preliminaries and basic definitions}

In this section we set our background. We begin by recalling of useful properties of the asymptotic variety, then we recall the concept of the intersection homology and we finish by recalling the $L^\infty$ homology.
 
\subsection{Notations and conventions.}

Given a subset $X$ of $\R^n$ we denote by $C_i(X)$ the group of
$i$-dimensional singular chains with the coefficients in $\R$, if
$c$ is an element of $C_i(X)$, we denote by $|c|$ its support.
By $Reg(X)$ and $Sing(X)$ we denote respectively the regular and singular locus of the set $X$.
\subsection{The Jelonek set.}
For a polynomial map $F:\K^n\to \K^m$, we denote by $J_F$ the set
of points at which the map $F$ is not proper, i.e.
$J_F=\{y\in\K^m\,\,\,\mbox{\it such that}\,\,\, \exists
\{x_k\}\subset \K^n, \,\,|x_k|\to\infty,\,\, F(x_k)\to y\}$, and
call it the asymptotic variety or the Jelonek set of $F$. The
geometry of this set was studied  by Jelonek in the series of
papers \cite{j1,j2,j3}. He obtained a nice description of this set,
gave also an upper bound of the degree of  this set. For the
details and the applications of these results we refer to the
works of Jelonek.  In our paper we will need the following
powerful theorem.
\begin{thm}[Jelonek \cite{j2}]\label{jelonek}
Let $F:\C^n\to\C^m$ be generically finite polynomial map then the
set $J_F$ is a $n-1$ pure dimensional $\C$-unirulled algebraic
variety or an empty set.
\end{thm}Here by $\C$ -
unirulled variety we mean that through any point of this variety passes a
rational curve included in this variety ($X$ is $\C$-unirulled if for all $x\in X$ there exists a non-constant polynomial map $\varphi_x:\C\to X$ such that $\varphi_x(0)=x$).

In the real case the Jelonek set is an $\R$-unirulled
semi-algebraic set but there is no restriction on its dimension
just it can be any number between 1 and $n-1$ \cite{j3}.

\subsection{Intersection homology.}
We briefly recall the definition of intersection homology, for details
we refer  to the fundamental work of Goresky and MacPherson
\cite{gm1}.
\begin{dfn}\label{del pseudomanifolds}A {\bf pseudomanifold} is a subset of $\R^n$ whose singular locus
is of codimension at least 2.
\end{dfn}

All the pseudomanifolds that we will consider in
this paper will be semi-algebraic.

Let $X$ be a pseudomanifold, consider a filtration of subsets
$\emptyset=X_{-1}\subset X_0\subset\dots\subset X_m=X$ such that
$X_i\setminus X_{i-1}=\emptyset$ or $\dim(X_i\setminus
X_{i-1})=i$. We will say that $X$ is a {\bf stratified
pseudomanifold} if the stratification given by the sets $X_i\setminus
X_{i-1}$, for $0\leq i\leq m$ is locally topologically trivial
along the strata.


\begin{dfn}
A perversity is a $(m-1)$-tuple  of integers $\bar p = (p_2,
p_3,\dots , p_m)$ such that $p_2 = 0$ and $p_{k+1}\in\{p_k, p_k + 1\}$.

Traditionally we denote the zero perversity by
$\overline{0}=(0;\dots;0)$, the maximal perversity by
$\overline{t}=(0;1;\dots;m-2)$ and the middle perversities by
$\overline{m}=(0;0;1;1;\dots [\frac{m-2}{2}])$ (lower middle) and
$\overline{n}=(0;1;1;2;2;\dots [\frac{m-1}{2}])$ (upper middle). We say that the
perversities $\pb$ and $\qb$ are {\bf complementary} if $\pb+\qb=\tb$.

Let $X$ be a stratified pseudomanifold and  let $\sigma_j:T_i\to X$ be a singular simplex. We say that the support of a
singular chain $c=\sum\alpha_j\sigma_j$ of dimension $i$ is {\bf $(\bar
p, i)$-allowable} if  $\dim (|c| \cap X_{m-k}) \leq i - k + p_k$ for
all $k \geq 2$.

In particular, a subspace $Y \subset X$ is called   $(\overline{t};i)$-allowable if $\dim( Y \cap X_{sing}) <i-1$.

Define $IC_i ^{\overline{p}}(X)$ to be the subgroup of $C_i(X)$
consisting of those chains $\xi$ such that $|\xi|$ is
$(\overline{p}, i)$-allowable and $|\partial \xi|$ is
$(\overline{p}, i - 1)$-allowable.
\end{dfn}

\begin{dfn} The {\it $i^{th}$ intersection homology group of perversity $\overline{p}$}, denoted by
$IH_i ^{\overline{p}}(X)$, is the $i^{th}$ homology group of the
chain complex $IC^{\overline{p}}_\bullet(X).$\end{dfn}

 Goresky and MacPherson proved that these groups are independent of the
choice of the stratification and are finitely generated \cite{gm1, gm2}.

Recall also a remarkable \begin{thm}[Goresky,
MacPherson \cite{gm1}] Let $X$ be an orientable, compact
pseudomanifold then generalized Poincar\'e duality applies:
 \begin{equation}\label{eq poincare duality} IH_k ^{\pb}(X)  \simeq IH_{m-k} ^{\qb} (X), \end{equation}
where $\pb$ and $\qb$ are complementary perversities.\end{thm}
In the non-compact case the above isomorphism holds
for Borel-Moore homology:
\begin{equation}\label{eq poincare duality} IH_k ^{\pb}(X)  \simeq IH_{m-k,BM} ^{\qb} (X), \end{equation}
where $IH_{\bullet,BM}$ denotes the intersection holomology with
respect to Borel-Moore chains \cite{gm2,k}.

\subsection{$L^\infty$ cohomology}
 Let $M\subset \R^n$ be a smooth submanifold.

\medskip

\begin{dfn}  We say that a form  $\omega$ on $M$
is $L^\infty$ if there exists a constant $C$ such that  for any $x
\in M$: $$|\omega(x)| \leq C. $$
 We denote by $\Omega^j_\infty
(M)$ the cochain complex constituted by all the $j$-forms
$\omega$ such that $\omega $ and $d\omega$ are both $L^\infty$.

The cohomology groups of this cochain complex are called the {\bf
$L^\infty$-cohomology groups of $M$} and will be denoted by
$H^\bullet _\infty(M)$. We may endow this cochain complex with the
norm:$$|\omega|_\infty:=\sup_{M} |\omega|+\sup_M |d\omega|. $$
\end{dfn}

Recently the second author showed that   the $L^\infty$
cohomology of a pseudomanifold coincides with its intersection
cohomology in the maximal perversity (\cite{v}, Theorem 1.2.2):
\begin{thm}\label{thm_intro_linfty} Let $X$ be a compact subanalytic pseudomanifold (possibly with boundary). Then, for any $j$:
$$H_\infty ^j(Reg(X)) \simeq I^{\bar t}H^j (X).$$
Furthermore, the isomorphism is induced by the natural map
provided by integration on allowable simplices.
\end{thm}

The theorem presented in the latter paper was devoted to pseudomanifolds without boundary but actually still aplies when the boundary is nonempty (as mentioned in the introduction of \cite{v}).

%

\section{Construction of an embedding}
For now we will consider a polynomial map $F:\C^n\to\C^n$ as a
real one $F:\R^{2n}\to\R^{2n}$. By $Sing(F)$ we mean a singular
locus of $F$ which is the zero set of its Jacobian and we denote by
$K_0(F)$ the set of critical values of $F$ i. e. the set  $F(Sing(F))$.

  We denote by $\rho$
the Euclidean Riemannian metric of $\R^{2n}$. We can pull it back 
in a natural way:
$$F^\ast\rho_x(u,v):=\rho(d_xF(u),d_xF(v)).$$ This metric
 is non degenerate outside the singular locus of $F$.

Define the Riemannian manifold $M_F:=(\C^n\setminus
Sing(F);F^\ast\rho)$, and observe that  the map
$F$ induces a local isometry nearby any point of $M_F$.

The first important step is to embed the manifold $M_F$ into an affine
space. For this we can make use of $F$.  The only problem is that
$F$ is not necessarily one-to-one, but just locally one-to-one. We
begin by  proving the following lemma, which is a well
known fact that we detail for the convenience of the reader.
\begin{lem}\label{lem F local homeo} There exists a finite  covering of  $M_F$ by open
semi-algebraic subsets, such that on any element of this covering
the mapping $F$ induces a diffeomorphism.
\end{lem}
\begin{proof}
Let $\Gamma_F \subset \R^{2n} \times \R^{2n}$ be the graph  of the
restriction of $F$ to $M_F$.  As the restriction of $F$ to $M_F$
is finite-to-one, there exists a partition $(C_i)_{i \in I}$ of
$\R^{2n}$, into smooth semi-algebraic cells, such that the set
$\Gamma_F$ is, over each open $C_i$, the union of the respective graphs
of finitely many semi-algebraic smooth mappings
$\xi_i^1,\dots,\xi_i ^{s_i}:C_{i} \to \R^{2n}$. Take an open cell
$C_i$. Every point of its boundary possesses a basis of neighborhoods $U$ in
$\R^{2n}$ such that $U \cap C_i$ is connected. At any point of the
boundary of $C_i$ which does not belong to $K_0(F)$, the function $\xi_i ^j$, $1 \leq j \leq s_j$ may
have only one asymptotical value in $\R^{2n} \cup \{\infty\}$
(again due to the finiteness of the fibers of $F$ and
connectedness of $C_i\cap U$). This means that the mapping $\xi_i$
extends continuously to the boundary of $C_i$  to a mapping with
values in $\R^{2n} \cup \{\infty\}$. Let $D_i$ be the graph of
$\xi_i$ restricted to the set  $\{x \in {\overline (C_i)}\setminus  K_0(F)\,\,|\,\, \xi_i(x) <\infty\}$. This set is clearly closed in $\Gamma_F $.

Let $\pi_1:\C^{2n} \to \C^n$ (resp. $\pi_2$) be the projection on
the first (resp. second) $n$ coordinates. As $F$ is nonsingular on
$M_F$, for any $x \notin Sing(F)$ there exists a neighborhood $U$
of $x$ such that $F_{|U}$ is an analytic isomorphism onto its
image. This means that $\pi_2$ induces a local analytic
isomorphism in a neighborhood of $(x;F(x))$ in $\Gamma_F$ for $x
\notin Sing(F)$. Hence, for any $z $
  in  $\pi_2(D_i)$, the mapping $\xi_i^j$ may be extended to a smooth
mapping in a neighborhood of $\pi_2(z)$ and therefore $F$ induces
a diffeomorphism from a neighborhood of $\pi_1 (D_i)$ in $M_F $
onto its image.
\end{proof}

The next proposition will enable us to transfer the geometry at infinity of
given  polynomaial map $F:\C^2\to\C^2$ to the constructed set. Namely the
intersection homology of the set $N_F$ provided by following
proposition determines the geometry of $F$ at infinity as we will
see in the main theorem.
\begin{pro}
 Let $F:\C^n\to\C^n$ be a polynomial map.
There exists a real semi-algebraic pseudomanifold $N_F \subset
\R^{\nu}$, for some $\nu \geq n$, such that \begin{equation}Sing(N_F)\subset (J_F\cup
K_0(F))\times\{0\}\end{equation}
 and there exists a semi-algebraic bilipschitz map:
$$h_F:M_F \to Reg(N_F),$$
where $N_F$ is equipped with the induced metric on $\R^{\nu}$.

\end{pro}
\begin{proof}
Let $U_1,\dots,U_p$ be the open sets provided by Lemma \ref{lem F
local homeo}. We may find some semi-algebraic closed subsets $V_i
\subset U_i$ which cover $\R^{2n}$ as well.

Thanks to Mostowski separation theorem \cite{m}, see also Lemma 8.8.8 in \cite{bcr}, there exist Nash functions $\psi_i:\R^{2n} \to \R$,
$i=1,\dots,p$, such that for each $i$,  $\psi_i$ is positive on
$V_i$ and negative on $\R^{2n}\setminus U_i$. Define:
\begin{equation}\label{h_F}h_F:=(F;\psi_1;\dots;\psi_p)\,\,\mbox{and}\,\, N_F:=\overline{h_F(M_F)}.\end{equation} 

We first check that the mapping $h_F$ is injective on $M_F$. Take $x \neq x'$
with $h_F(x)=h_F(x')$, and let $i$ be such that $x\in V_i$. As
$F(x')=F(x)$, the injectivity of $F$ on $U_i$ entails that $x'
\notin U_i$. But this means that $\psi_i(x)>0$ and $\psi_i(x')<0$.
A contradiction.

We  claim that if the functions $\psi_i$ are chosen
sufficiently small then the mapping $h_F$ is 
bilipschitz. For this we will use \L ojasiewicz inequality in the following form:
\begin{pro}\label{lojasiewicz}
Let $A\subset\R^n$ be a closed semi-algebraic set, and $f:A\to\R$ a continuous semi-algebraic function. Then there exist 
$c\in\R$ and $p\in\N$ such that for any $x\in A$ one have
$$|f(x)|\le c(1+|x|^2)^p.$$
\end{pro}
For the proof see \cite{bcr}, Proposition 2.6.2.

 Now take a point $z$ in $N_F$ and denote by
$\pi_1:\R^{2n+p} \to \R^{2n}$ the projection on the first $2n$
coordinates. Then $y=\pi_1(z)$ belongs to the image of $F$. Choose
$x \in \R^{2n}$ with $F(x)=y$ and fix a neighborhood of $x$ such
that the mapping $F_{|U}:U \to \R^{2n} $ is a diffeomorphism onto
its image.  Define, for $y \in F(U)$ following functions:
\begin{equation}
\tilde{\psi}_i(y):=\psi_i \circ (F_{|U})^{-1}(y)
\end{equation}
and 
\begin{equation}
\hat{\psi}(y):=(y;\tilde{\psi_1}(y);\dots;\tilde{\psi_p}(y)),
\end{equation}
for $i=1,\dots ,p$.

We have then formula
\begin{equation}h_F(x)=(F(x);\tilde{\psi}_1(F(x)));\dots;\tilde{\psi}_p(F(x)))=\hat{\psi} (F(x)).\end{equation}

As the map $F: (U;\rho_F) \to F(U)$ is bilipschitz, it is enough to
show that $\hat{\psi}:F(U) \to \pi_1^{-1} (F(U))$ is 
bilipschitz. This amounts to prove that $\tilde{\psi}_i$ has
bounded derivatives for any $i=1,\dots,p$. By \L ojasiewicz
inequality (Proposition \ref{lojasiewicz}) there exist  a positive constant $\varepsilon$ and an
integer $N$ such that $$ \underset{u \in S^n}{\inf} |d_x F(u) | \geq
\varepsilon \frac{1}{(1+|x|)^N},$$ for any $x \in \R^{2n}$.
Furthermore the derivatives of $\psi_i$ satisfy a similar
inequality as well. Therefore multiplying $\psi_i$ by a huge power
of $\frac{1}{1+|x|^{2}}$ (which is a Nash function) if necessary
we see that we can assume that the derivatives of $\tilde{\psi}_i$
are bounded, as required.

To finish the proof, it remains to show that $N_F\setminus
h_F(M_F)$ is containedi in  the set $$(J_F\cup
K_0(F))\times\{0\}.$$ Again thanks to  \L ojasiewicz inequality, we may
assume that the functions $\psi_i$ tend to zero at infinity and
near $Sing(F)$. Consequently, $N_F\setminus
h_F(M_F)$ coincides with the Jelonek set of $F$
together with $K_0(F)$, which is of codimension $2$ in $\R^{2n}$
thanks to Theorem \ref{jelonek}.
\end{proof}

%
%
%

\begin{rem}
From the proof we get immediately that the following diagram is
commutative:
\begin{center}
     \begin{picture}(-140,0)
      \put(-160,-50){$\R^{2n}$}
        \put(-88,-5){$Reg(N_F$)}
          \put(-68,-50){$\R^{2n}$}
     \put(-140,-46){\vector(3,0){65}}
  \put(-136,-35){\vector(2,1){50}}
      \put(-60,-10){\vector(0,-1){30}}
     \put(-57,-28){$\pi_F$}
\put(-138,-28){$h_F$}
     \end{picture}
    \end{center}
\vskip2cm

Where $\pi_F$ is locally bilipschitz.
\end{rem}

\begin{rem} If the Jacobian of $F$ is identically equal to $1$, then every point $x$ has a neighborhood ßUß such that the restriction of $F$ to this neighborhood preserves the Lebesgue measure. In this case the mapping $h_F$  preserves the volume of subsets of $U$ ``up to constant'' in the sense that there exist constants $C_1$ and $C_2$ such that for any subset $A$ of $U$ we hqve:$$C_1\mathcal{L}^n(A)\leq \mathcal{H}^n(h_F(A)) \leq C_2\mathcal{L}^n(A),$$
where $\mathcal{H}^n$ and $\mathcal{L}^n$ denote respecively the $n$-dimensional Hausdorff and Lebesgue measure.

\end{rem}


%
%

\begin{lem}
Let $F:\C^2\to\C^2$ be a polynomial map then there exists a
natural stratification of the set $N_F$, with only even dimensional strata, 
which is topologically trivial along the strata.
\end{lem}
\begin{proof}
 Consider the set
$$A:=\pi_F^{-1}(J_F)\setminus J_F \times \{0_{\R^{\nu-4}}\}.$$
We first check the following fact:

\noindent {\bf Claim.} The set $B:=cl(A)\cap J_F \times \{0_{\R^{\nu-4}}\} $ is finite.

It is enough to show that it is included in the set of asymptotical values of $F_{|F^{-1}(J_F)}$ (for $F^{-1}(J_F)$ is a complex algebraic curve). Let $x \in B$. There exists a (real analytic) curve un $\alpha$ in $A$ tending to $x$ . Clearly $\pi_F(\alpha(s)) \in J_F$ entails $h_F^{-1}(\alpha(s)) \in  F^{-1}(J_F) $.  As $\alpha$ ends at a point of the singular locus of $N_F$ and $h_F$ is one-to-one and  nowhere singular, the preimage of $\alpha$ by $h_F$ must go to infinity, which means that $\pi_F(x)=x$ is an asymptotical value of $F_{|F^{-1}(J_F)}$.

Let now $$N_0=B \cup Sing(J_F\cup K_0(F)), $$

and take $x_0 \in (J_F\cup K_0(F))\setminus  N_0$.

We have to show that every point $x_0$, has a  neighborhood in $N_F$ which is topologically trivial.  As we will also show that this topological trivialization preserves $J_F$  this will show that the stratification given by $(J_F\cup
K_0(F);N_0)$ is locally topologically trivial along the strata.

Since the problem is purely local we may identify $J_F\times \{0\}$ with $\R^{2}\times \{0\}$  and  
work nearby the origin.  Consider now the following  local isotopies
$$ \mu_i :U \times [-\varepsilon ;\varepsilon] \to U,$$
where $U$ is a neighborhood of $0$ in $\R^4$, defined by  $\mu^i _i(x;t)=x+t e_i$, $i=1,2$. Choosing $U$ small enough, we can assume $U\cap B=\emptyset$.  Over the complement of $(K_0(F)\cup J_F)\cap U$, these isotopies may be lifted to isotopies on $\pi_F^{-1}(U)$ for $\pi_F$ induces a covering map.

We denote this lifting $\tilde{\mu}_i$.  On $U \cap (J_F\cup K_0(F))\times \{0\}$, $\pi_F(x;0)=x$, so that each $\tilde{\mu}_i$ extends continuously. As $\pi_F$ is  a local homeomorphism at any point of $$Reg(N_F\setminus ((K_0(F)\cup J_F) \times \{0\}), $$ each $\tilde{\mu}_i$  extends continuously when $x$ tends to any point $\pi_F^{-1}(U\cap J_F)$. The obtained isotopies may not fall into $J_F \times \{0\}$ since $U$ does not meet the set $B$.

\end{proof}

\section{Non properness and vanishing cycles.} In the case $n=2$, the singular homology as well as the intersection homology of the constructed set $N_F$ captures the behavior at infinity of $F$. If $F$ is proper then $N_F $ is a ball in  $\R^4$ and thus $$H_2(N_F)=IH_2 ^{\overline{t}} (N_F)=0.$$
We are going to show that the converse is also true (Theorem \ref{thm homology propre}).

For this purpose we start by proving

\begin{lem}\label{lem sigma borde qqch fibre en 0}
Let $\alpha:T_j \to \R^m$ be a singular  $j$-chain and let $A \subset \R^m \times \R$ be a compact  semi-algebraic family of sets with $|\beta|\subset A_t$ for any $t$. Assume that $\beta$ bounds a $(j+1)$ in each $A_t$, $t>0$. Then $\beta$ bounds a chain in $A_0$.
\end{lem}
\begin{proof}  Given a compact family of  sets $A \subset \R^n \times
[0;\varepsilon]$, there exists a  semi-algebraic  triangulation of $A$ such that $A_0$ is a union of image of simplices (here $A_0$ denotes the zero fiber of the family $A$). Therefore, there exists a  continuous strong deformation
retraction of a neighborhood $U$ of $A$ onto $A_0$.  As $A_t$ lies in $U$ for $t$  small enough, the lemma ensues.
\end{proof}

\begin{thm}\label{thm homology propre}
Let $F:\C^2\to\C^2$ be a polynomial map with nowhere vanishing
Jacobian.  Fix a perversity $\bar p$. The following conditions are equivalent:
\begin{enumerate}
\item $F$ is non proper, 

\item $H_2(N_F)\neq 0,$

\item $IH^{\overline{p}} _{2} (N_F) \neq
 0.$



 \end{enumerate}
\end{thm}
\begin{proof} We first show that $(1)$ implies $(2)$.  As the cycle that we will exhibit will be $(\bar{p};2)$ allowable for any perversity $\bar{p}$, this will also establish that  $(1)$ implies $(3)$.

  Assume that $F$ is not proper.  This means that there exists a Puiseux arc $\gamma:\mathbb{D}(0;\eta) \to \R^4$,
 $$\gamma(z)=a z^\alpha +\dots, $$
 (with $\alpha$ negative and $a \in \R^4$  unit vector) tending to infinity in such a way
that $F(\gamma)$ converges to an element $x_0 \in \R^4$.

Let $C$ be an oriented triangle in $\R^4$ whose barycenter is the origin. Then, as $h_F \circ \gamma$ extends continuously at $0$, the map  $h_F\circ \gamma$, provides a singular $2$-simplex in $N_F$ that we will denote by $c$. This simplex is $(\overline{t};2)$ allowable since $$J_F \cap |c|=\{x_0\}.$$ The support of $\partial c$ lies in $N_F \setminus J_F$. As, by definition of $N_F$,  $N_F \setminus J_F \simeq \R^4$, this means that $\partial c$ bounds a singular chain $e \in C_3(N_F \setminus J_F)$. But then $\sigma=\partial c-e$ is a $(\overline{t};2)$-allowable cycle of $N_F$. We claim that $\sigma$ may not bound a  $3$-chain in $N_F$.

 Assume otherwise, it means that  there exists a  chain $\tau \in C_3(N_F)$, satisfying $\partial \tau=\sigma$.



For $R$ large enough, the sphere $S(0;R)$ is transverse to $A$ and $B$ (at regular points). Therefore, after a triangulation, the intersection $\sigma_R:=S(0;R)\cap A$ is a  chain bounding the chain $\tau_R:=S(0;R) \cap B$.  

Given a subset $A\subset \R^{4}$ we define the ``tangent cone at
infinity" by: $$C_\infty (A):= \{\lambda \in S^3 : \exists
\gamma:(0;\varepsilon] \to A \: \mbox{Nash},\: \lim_{t \to 0}
\gamma(t) =\infty, \: \lim_{t \to 0}
\frac{\gamma(t)}{|\gamma(t)|}=\lambda\}.$$

Let $\hat{F}_1$ and $\hat{F}_2$ be the initial forms of the
components of $F:=(F_1;F_2)$ and let $\alpha$ be any analytic arc in $\R^4$, tending to infinity and such that $F(\gamma)$ does not tend to infinity. Then,  for any $i=1,2$, the
Puiseux expansion of  $\hat{F}_i(\gamma(z))$ must start like:
\begin{equation}\label{eq expansion de F_i}\hat{F}_i(\gamma(z))=F_i(a)
z^\beta+\dots,\end{equation} for some negative integer $\beta$. As
the arc $F(\gamma)$ does not diverge to infinity as $z$ tends to
zero, we see that
$$F_i(a)=0, \qquad \qquad i=1,2.$$

Therefore, if a set $X\subset \R^4$ is mapped onto a bounded set by $F$ then  $C_\infty(X)$ is included  in the zero
locus $V$ of $\hat{F}:=(\hat{F}_1;\hat{F}_2)$. Hence $C_\infty(A)$ and $C_\infty(B)$ are subsets of $V\cap S^3$. Observe that, in a neighborhood of infinity,  $A$ coincides with the support of the complex analytic arc $\gamma$. Thus $C_\infty(A)$ is nothing but $S^1\cdot a$.





Consider  a Nash strong deformation retraction $\rho:W \times[0;1]  \to S^1 \cdot a$, of a
neighborhood $W$ of $\C a \cap S^{3}$   in the sphere onto $S^1\cdot a$.

Let $\tilde{\sigma}_R=\frac{\sigma_R}{R}$. The image of the restriction of $\gamma$ to a sufficiently small
neighborhood of $0$ in $\C$ entirely lies in the cone over $W$. Consequently, $|\tilde{\sigma}_R|$ is included in $W$, for $R$ large enough.

Let
$\sigma' _R$ be the image of $\tilde{\sigma}_R$ under the retraction $\rho_1$ (where $\rho_1(x):=\rho(x;1)$).
As, near infinity,  $|\sigma_R|$ coincides with the support of the arc $\gamma$,  for $R$ large enough the class of $\sigma'_{R} \cap
S^{3} (0;R)$ is nonzero is $S^1 \cdot Ra$.

 As the retraction $\rho_1$ is isotopic to the identity,
there exists a chain $\theta \in $ such that:
$$\partial \theta_R=\sigma'_R-\tilde{\sigma}_R.$$

The sets  $|\theta_R|$ constitute a semi-algebraic family of sets. Denote by  $E \subset \R^4 \times \R$ its closure and set $E_0:=E \cap \R^4 \times 0$. As the strong deformation  retract $\rho$ is the identity on $S^1 \cdot a \times [0;1]$, we see that  \begin{equation}\label{eq cone de E} E_0\subset S^1 a \subset V\cap S^3. \end{equation}

Let now $\tilde{\tau}_R =\dfrac{\tau_R}{R}$ and:
$$\theta'_R:=\tilde{\tau}_R +\theta_R,$$
and denote by $E'$ the semi-algebraic family of sets constituted the closure of the union of all by the sets $|\theta'_R|$. Again denote by $E'_0$ the fiber at $0$, i. e. the set $E' \cap\R^4 \times \{0\}$.
By (\ref{eq cone de E}) and the definition of $\theta'$ we have \begin{equation}\label{eq theta'' et tau} E' _0\subset
E_0 \cup
C_\infty(B) \subset V\cap S^3,\end{equation} and $\partial \theta'_R
=\sigma'_R$.

The class of $\sigma'$   is, up to a product by a nonzero constant,
equal to the generator  of $H^1 (S^1 \cdot a)$.
It means that the cycle $S^1\cdot a$ bounds the  chain $\theta'_R$ which has support in  $E'_R \subset V\cap S^3$.
By Lemma \ref{lem sigma borde qqch fibre en 0} this implies that $S^1\cdot a$ bounds a chain in $E'_0 \subset V\cap S^3,$ a contradiction. 

If $F$ is proper map, then the set $N_F$ is homeomorphic to $\C^2$
and consequently $H_2(N_F)=IH^{\tb}_2(N_F)=0$. Thus, $(2)$ and $(3)$  both fail.

\end{proof}

\begin{obs}
The set $N_F^R:=N_F\cap \bar B (0,R)$ is a pseudomanifold with boundary
for large enough $R$.
\end{obs}
\begin{proof}
We are going to show that $N_F^R$ is a stratified pseudomanifold.
First of all,  observe that if $R$ is chosen large enough then $R$
is not a critical value of the distance function to the origin on the set $N_F$ and thus  $N_F ^R$ is a smooth
manifold with boundary at any point of $ N_F  \cap \{\rho=R\}
\setminus J_F  $. Furthermore, by Hardt's Theorem \cite{h}, the level surfaces  $N_F \cap
\{\rho=R\}$  are pseudomanifolds  constituting a semi-algebraically
topologically trivial family.  The
trivialization may be required to preserve $J_F$. Consequently, the set
$$(N_F^R;N_F \cap \{\rho=R\})$$ is a pseudomanifold with boundary.
\end{proof}

Now thanks to the de Rham theorem for $L^\infty$ forms (Theorem \ref{thm_intro_linfty}) we get the following immediate corollary.

\begin{cor}
\label{cor homology propre}
Let $F:\C^2\to\C^2$ be a polynomial map with nowhere vanishing
Jacobian.  The following conditions are equivalent:
\begin{enumerate}\item $F$ is nonproper,
 \item $H^2_\infty(Reg(N_F^R))\neq 0$.
 \end{enumerate}
\end{cor}




\end{document}